\documentclass[12pt,leqno]{amsart}%
\usepackage{amssymb}
\usepackage{amscd}
\usepackage{amsmath}
\usepackage{amsfonts}
\usepackage{graphicx}%
\newtheorem{theorem}{Theorem}
\newtheorem{corollary}[theorem]{Corollary}
\newtheorem{example}[theorem]{Example}
\theoremstyle{definition}
\theoremstyle{remark}
\numberwithin{equation}{section}

\begin{document}
\title{Some nonlinear second order equation modelling rocket motion}
\author{Dorota Bors and Robert Sta\'{n}czy}
\address{\noindent Dorota Bors, Departament of Mathematics and Computer Science, \newline
University of \L \'{o}d\'{z} \\
S. Banacha 22\\
90-238 \L \'{o}d\'{z}; \newline \indent
Robert Sta\'nczy, Instytut Matematyczny\\
Uniwersytet Wroc\l awski,\newline
Pl. Grunwaldzki 2/4\\
50-384 Wroc\l aw;}
\email{bors@math.uni.lodz.pl, stanczr@math.uni.wroc.pl;}
\date{}
%\thanks{This work has been partially supported by the Polish Ministry of Science project N N201 418839}

\begin{abstract}
In this paper, we consider a nonlinear second order equation modelling rocket
motion in the gravitational field obstructed by the drag force. The proofs of
the main results are based on topological fixed point approach.

\end{abstract}

\maketitle

\section{Introduction}

In this paper we consider the following second order differential equation 
\begin{equation}
-\ddot{z}\left(  t\right)  =\alpha\left(  t\right)  \left(  \dot{z}\left(
t\right)  +a\right)  ^{2}\exp\left(  -\frac{z\left(  t\right)  }{H}\right)
+\beta\left(  t\right)  \label{BVP}%
\end{equation}
with suitably chosen functions $\alpha, \beta$ to be defined below by (\ref{albe}) and some constants $a, H$ to be determined later and suitable initial or boundary conditions derived from some model describing the motion of the rocket, cf. \cite{Mes}. Specifically, it can be derived from the one-dimensional first oder differential equation
\begin{equation}
m\left(  t\right)  \dot{v}\left(  t\right)  =\left(  w\left(  t\right)
-v\left(  t\right)  \right)  \dot{m}\left(  t\right)  +f^{ext}\left(
t\right).  \label{R1}%
\end{equation}
This equation can be used to model a rocket motion with decreasing mass $m=m\left(
t\right) $ and under external force $f^{ext}\left(
t\right)$. The rocket should be equipped with the engine that emits mass
$m(t_0)-m\left(  t\right)$ with velocities $w=w\left(  t\right)  ,$
$t\in\left[  t_{0},t_{1}\right]  .\,$Therefore, the sum of the
mass of the rocket $m\left(  t\right)  $ and the mass of the gas $m(t_0)-m(t)$ emitted at
time $t$ is constant and equal $m(t_0)=m_0$ which is some positive value. 
In $m_{0}$ we include all components of
the rocket infrastructure like the engine, empty tank, guidance equipment,
payload and the like. The function $f^{ext}=f^{ext}\left(  t\right) $ denotes all 
external forces that stimulate the rocket motion for example the gravity
force or the drag force and in fact may depend also on the position and the velocity
of the rocket, cf. (\ref{fex})

The expression $m\dot{v}(t)$ is the force acting on the rocket body which
is given by the amount of mass expelled and the velocity of the mass relative
to the rocket body, i.e. $(w(t)-v(t))\dot{m}(t)$. 
The rocket moves forward since the rocket loses mass as it
accelerates, so $\dot{m}(t)$ must be negative for any rocket. The term
$(w-v)\dot{m}$ is called the thrust
of the rocket and can be interpreted as an additional force on the rocket due
to the gas expulsion. The detailed derivation of the rocket equation one can
find, among others, in \cite{Kos, Mes, Par}.\newline The
equation $\left(  \ref{R1}\right)  $ is so called the Me\v{s}\v{c}erskii
rocket equation, for details see \cite{Mes,Kos}, where alos the case of more
that just one emitted mass was considered. If we assume, the external force to 
be zero and a constant relative velocity of the emitted mass, i.e. $f^{ext}=0$ and $w\left(
t\right)  -v\left(  t\right)  =:c<0$ for $t\in\left[  t_{0},t_{1}\right]  $,
then the equation $\left(  \ref{R1}\right)  $ has the form of the well-known
Tsiolkovskii rocket equation. \newline The corresponding solution to the
Tsiolkovskii equation with the initial conditions $m\left(  t_{0}\right)
=m_0$ and $v\left(  t_{0}\right)  =v_0$ can be easily found after
integration of $\left(  \ref{R1}\right)  $ and reads as follows%
\[
v\left(  t\right)  =v_0+c\ln\frac{m\left(  t\right)  }{m_0}.
\]
The Tsiolkovskii rocket equation in the gravitational field, i.e.
$f^{ext}=-mg$ with the initial conditions $m\left(  t_{0}\right)  =m_0$
and $v\left(  t_{0}\right)  =v_0\ $has a solution of the form%
\[\label{Tso}
v\left(  t\right)  =v_0-g\left(  t-t_{0}\right)  +c\ln\frac{m\left(
t\right)  }{m_0}.
\]

Except for the rocket equation, the Tsiolkovskii equation can also be applied
to the one photon decay of the excited nucleus. Moreover, it is possible to
identify the Me\v{s}\v{c}erskii equation with the bremsstrahlung equation
where the electron is in the permanent decay. Many other decays can be
investigated from the point of view of the Me\v{s}\v{c}erskii equation, as one
can see in \cite{Par}.

Assume that $f^{ext}=-mg-D$ with $D=\frac{1}{2}\rho v^{2}AC_{D}$ where $C_{D}$
is a drag coefficient, $A$ is the cross sectional area of the rocket, $\rho$
is the air density that changes with altitude $x,$ and can be approximated by
$\rho=\rho_{0}\exp\left(  -\frac{x}{H}\right)  $ where $H\approx8000m$ is the
so called "scale height" of the atmosphere, and finally $\rho_{0}$ is the air
density at sea level. In that case it is impossible to solve the equation
explicitly. However, it is interesting to note that the effect of drag loss is
usually quite small and is often reasonable to ignore it in the first
approach. In practice, the drag force is approximately only about 2\% of the
gravity force.

In the paper we will consider the equation $\left(  \ref{R1}\right)  ,$ with a
given relative velocity of emitted masses $c\left(  t\right)
=w\left(  t\right)  -v\left(  t\right)  $ and
$t\in\left[  t_{0},t_{1}\right]  .$ Chemical rockets produce exhaust jets at
velocity $\left\vert c\right\vert $ from $2500$ $m/s$ to $4500$ $m/s.$
If we assume that the relative velocity is a given function $c\left(
t\right)  =w\left(  t\right)  -v\left(  t\right)  ,$ we obtain the following
equation%
\[
m\left(  t\right)  \dot{v}\left(  t\right)  =c\left(  t\right)  \dot{m}\left(
t\right)  +f^{ext}\left(  t\right)  .
\]
Let us put $v\left(  t\right)  =\dot{x}\left(  t\right)  $ where $x\left(
t\right)  $ denotes rocket position at time $t$. Furthermore, we assume that
the external force depends not only on time $t$ but also on the rocket
position $x\left(  t\right)  $ and its velocity $v\left(  t\right)  $
according to gravitational and drag forces, i.e.
\[\label{fex}
f^{ext}\left(  t\right)  =f\left(  t,x\left(  t\right)  ,v\left(  t\right)
\right)  =-m\left(  t\right)  g-\frac{AC_{D}\rho_{0}}{2}v^{2}\left(  t\right)
\exp\left(  -\frac{x\left(  t\right)  }{H}\right)
\]
such that $f:\left[  t_{0},t_{1}\right]  \times\mathbb{R}\times\mathbb{R}%
\rightarrow\mathbb{R}$.\newline Therefore, we get the equation
\begin{equation}
m\left(  t\right)  \ddot{x}\left(  t\right)  =c\left(  t\right)  \dot
{m}\left(  t\right)  -m\left(  t\right)  g-\frac{AC_{D}\rho_{0}}{2}\dot{x}%
^{2}\left(  t\right)  \exp\left(  -\frac{x\left(  t\right)  }{H}\right)
\label{R2}%
\end{equation}
for any $t\in\left[  t_{0},t_{1}\right]  .$\newline We provide our equation
$\left(  \ref{R2}\right)  $ with the boundary conditions%
\begin{equation}
x\left(  t_{0}\right)  =x_{0},\text{ }x\left(  t_{1}\right)  =x_{1},
\label{R3}%
\end{equation}
which guarantee that the rocket attains the given point in the required period
of time.

After a shift $z\left(  t\right)  =x\left(  t\right)  -y\left(  t\right)  $ by
a linear function%
\begin{equation}
y\left(  t\right)  =a\left(  t-t_{0}\right)  +x_{0}. \label{T}%
\end{equation}
where $a=\frac{x_{1}-x_{0}}{t_{1}-t_{0}}$ and equation $\left(  \ref{R2}\right)  $
reads
\begin{equation}
\ddot{z}\left(  t\right)  =c\left(  t\right)  \frac{\dot{m}\left(  t\right)
}{m\left(  t\right)  }-g-\frac{AC_{D}\rho_{0}}{2}\frac{\left(  \dot{z}\left(
t\right)  +a\right)  ^{2}}{m\left(  t\right)  }\exp\left(  -\frac{z\left(
t\right)  +y\left(  t\right)  }{H}\right)  \label{R4}%
\end{equation}
for any $t\in\left[  t_{0},t_{1}\right]  $ and $z$ satisfies the homogenous
Dirichlet boundary conditions%
\begin{equation}
z\left(  t_{0}\right)  =z\left(  t_{1}\right)  =0. \label{R4b}%
\end{equation}
Consequently, one can reduce considered equation to a aforementioned in (\ref{BVP}) form%
\begin{equation}\label{eqz}
-\ddot{z}\left(  t\right)  =\alpha\left(  t\right)  \left(  \dot{z}\left(
t\right)  +a\right)  ^{2}\exp\left(  -\frac{z\left(  t\right)  }{H}\right)
+\beta\left(  t\right)  %
\end{equation}
where%
\begin{align}
\alpha\left(  t\right)   &  =\frac{AC_{D}\rho_{0}}{2m\left(  t\right)  }%
\exp\left(  -\frac{y\left(  t\right)  }{H}\right)  ,\label{albe}\\
\beta\left(  t\right)   &  =g-c\left(  t\right)  \frac{\dot{m}\left(  t\right)
}{m\left(  t\right)  }.\nonumber
\end{align}
Note that since $\int_{t_0}^{t_1}\dot{x}(t)dt = x_1-x_0=a(t_1-t_0)$ we are led to search for the solutions $z$ 
to the boundary value problem $(\ref{R4b})-(\ref{eqz})$ such that $\int_{t_0}^{t_1}\dot{z}(t)dt=0$ and therefore
sign changing in the derivative, which can be interpreted as the acceleration or slowing down with respect to
the average speed corresponding to the linear motion described by the function $y$ defined by $(\ref{T})$ of the rocket moving with the constant average speed $a$.

\section{Integral formulation\label{Sec1}}

It should be noted that due to the presence of the derivative $\dot{z}$ the
equation $\left(  \ref{R4}\right)$, or equivalently $(\ref{eqz})$, is not in a variational form. Therefore
our approach differs from the one employed in \cite{BorWal}, as well as from \cite{Bor} where an extension to elliptic problems was cosidered,  and makes use of the integral formulation approach. Specifically, the boundary
value problem $\left(  \ref{BVP}\right)  $ with $\left(  \ref{R4b}\right)  $
can be reformulated in the following way%
\begin{equation}
z\left(  t\right)  =%
%TCIMACRO{\dint \limits_{t_{0}}^{t_{1}}}%
%BeginExpansion
{\displaystyle\int\limits_{t_{0}}^{t_{1}}}
%EndExpansion
G\left(  t,s\right)  F\left(  s,z\left(  s\right)  ,\dot{z}\left(  s\right)
\right)  ds+b\left(  t\right)  \label{Int}%
\end{equation}
where the nonlinear term $F$ is defined as
\[
F\left(  s,z,\dot{z}\right)  =\alpha\left(  s\right)  \left(  \dot
{z}+a\right)  ^{2}\exp\left(  -\frac{z}{H}\right)
\]
with the abuse of notation, namely $\dot{z}$ denotes in the above formula
the third independent variable of the function $F$ and the function $b$ is defined by%
\[
b\left(  t\right)  =%
%TCIMACRO{\dint \limits_{t_{0}}^{t_{1}}}%
%BeginExpansion
{\displaystyle\int\limits_{t_{0}}^{t_{1}}}
%EndExpansion
G\left(  t,s\right)  \beta\left(  s\right)  ds.
\]
Recall that $\alpha\left(  t\right)  $ and $\beta\left(  t\right)  $ are
defined by formulas $\left(  \ref{albe}\right)  .$ The nonnegative function $G$ is the
Green function for the BVP (\ref{R4})--(\ref{R4b}), 
i.e. the continuous symmetric function satisfying the following conditions.

\begin{enumerate}
\item[a.] the homogeneous equation, i.e. for any $t\neq s$%
\[
\frac{\partial^{2}G}{\partial t^{2}}\left(  t,s\right)  =0\,,
\]

\item[b.] the boundary conditions, i.e. for any $s\in\left[  t_{0},t_{1}\right]$
\[
G\left(  t_{0},s\right)  =G\left(  t_{1},s\right)  =0\,,
\]

\item[c.] the jump condition, i.e. for any $t\in(t_{0},t_{1})$
\[
\lim_{s\rightarrow t^{-}}\frac{\partial G}{\partial t}\left(  t,s\right)
-\lim_{s\rightarrow t^{+}}\frac{\partial G}{\partial t}\left(  t,s\right)  =1\,.
\]
\end{enumerate}

This function can be obtained from $\left(  \ref{BVP}\right)  $ by double
integration and use of the the boundary conditions $\left(  \ref{R4b}\right)
$ and it reads
\begin{equation}
G\left(  t,s\right)  =\left\{
\begin{array}
[c]{c}%
\frac{1}{t_{1}-t_{0}}\left(  t-t_{0}\right)  \left(  s-t_{1}\right)  \text{,
}t<s,\\
\frac{1}{t_{1}-t_{0}}\left(  s-t_{0}\right)  \left(  t-t_{1}\right)  \text{,
}t>s.
\end{array}
\right.  \label{Green}%
\end{equation}
with the derivative%
\begin{equation}
\frac{\partial G}{\partial t}\left(  t,s\right)  =\left\{
\begin{array}
[c]{c}%
\frac{1}{t_{1}-t_{0}}\left(  s-t_{1}\right)  \text{, }t<s,\\
\frac{1}{t_{1}-t_{0}}\left(  s-t_{0}\right)  \text{, }t>s.
\end{array}
\right.  \label{Greender}%
\end{equation}

\section{Main results}

We shall work in the space $C^{1}=C^{1}\left(  \left[  t_{0},t_{1}\right]
,\mathbb{R}\right)  $ of continuously differentiable functions equipped with
the natural norm%
\[
\left\Vert z\right\Vert =\max\left\{  \left\vert z\right\vert ,\left\vert
\dot{z}\right\vert \right\}
\]
where $\left\vert z\right\vert =\sup_{t\in\left[  t_{0},t_{1}\right]
}\left\vert z\left(  t\right)  \right\vert $ and $\left\vert \dot
{z}\right\vert =\sup_{t\in\left[  t_{0},t_{1}\right]  }\left\vert \dot
{z}\left(  t\right)  \right\vert $ for any $z\in C^{1}.$ Next we can denote
the right hand side of $\left(  \ref{Int}\right)  $ as the value of the
integral operator $Sz,$ i.e.%
\begin{equation}
Sz(t)=%
%TCIMACRO{\dint \limits_{t_{0}}^{t_{1}}}%
%BeginExpansion
{\displaystyle\int\limits_{t_{0}}^{t_{1}}}
%EndExpansion
G\left(  t,s\right)  F\left(  s,z\left(  s\right)  ,\dot{z}\left(  s\right)
\right)  ds+b\left(  t\right)  .\label{Sz}%
\end{equation}
Thus the existence of solution to $\left(  \ref{Int}\right)  $ is reduced to
finding a fixed point of the operator $S,$ i.e. such a $z\in C^{1}$ that
\[
Sz=z.
\]
First of all, note that the operator $S$ is well-defined, continuous and
compact. Indeed, for any $z\in C^{1}$ the function $Sz$ is continuous since
the functions appearing under the integral sign $\left(  \ref{Sz}\right)  $
are uniformly continuous and $\left(  Sz\right)  ^{\prime}$ is also continuous
by the integrability of $\frac{\partial G}{\partial t},$ cf. $\left(
\ref{Greender}\right)  .$ The continuity of the operator $S$ follows from the
smoothness of $F$ with respect both to $z$ and $\dot{z}$ and boundness of the
functions $G$ and $\frac{\partial G}{\partial t}.$ The compactness of the
operator $S$ requires the straightforward application of the Ascoli-Arz\`{e}la
theorem. 

\subsection{Main result. Sign insensitive case.}

In this subsection we provide crude but immediate estimates yielding by the
Schauder fixed point theorem the existence of a fixed point for the operator
$S.$ To this end we start with simple estimates of the sup norms $|\cdot|$ 
of the following functions%
\begin{align}
\left\vert Sz\right\vert  &  \leq G_{0}\left\vert \alpha\right\vert \left(
\left\vert \dot{z}\right\vert ^{2}+2a\left\vert \dot{z}\right\vert +a^2\right)
+\left\vert b\right\vert \label{estimates}\\
\left\vert \left(  Sz\right)  ^{\prime}\right\vert  &  \leq G_{1}\left\vert
\alpha\right\vert \left(  \left\vert \dot{z}\right\vert ^{2}+2a\left\vert
\dot{z}\right\vert +a^2\right)  +\left\vert \dot{b}\right\vert \nonumber
\end{align}
where
\begin{align*}
G_{0} &  =\sup_{t\in\left[  t_{0},t_{1}\right]  }%
%TCIMACRO{\dint \limits_{t_{0}}^{t_{1}}}%
%BeginExpansion
{\displaystyle\int\limits_{t_{0}}^{t_{1}}}
%EndExpansion
\left| G\left(  t,s\right) \right| ds,\\
G_{1} &  =\sup_{t\in\left[  t_{0},t_{1}\right]  }%
%TCIMACRO{\dint \limits_{t_{0}}^{t_{1}}}%
%BeginExpansion
{\displaystyle\int\limits_{t_{0}}^{t_{1}}}
%EndExpansion
\left|\frac{\partial G}{\partial t}\left(  t,s\right) \right| ds.
\end{align*}
Using formula $\left(  \ref{Green}\right)  $ one can easily derive the value
of
\[
G_{0}=\frac{3}{8}\left(  t_{1}-t_{0}\right)  ^{2}%
\]
and 
\[
G_{1}=\left(  t_{1}-t_{0}\right)  ^{2}.
\]

\begin{theorem}
The boundary value problem $\left(  \ref{BVP}\right)  -\left(  \ref{R4b}%
\right)  $ admits at least one solution provided the data are sufficiently small, to be more specific if
the condition $\left(  \ref{best}\right) $ is satisfied.
\end{theorem}

\begin{proof}
As announced we will make use of the Schauder fixed point theorem for the
operator $S$.defined by $\left(  \ref{Sz}\right)  .$ We should prove that it
maps some ball $B\left(  0,R\right)  $ in $C^{1}$ into itself. Using the
preliminary estimates $\left(  \ref{estimates}\right)  $ and defining
\[
G_{2}=\left\vert \alpha\right\vert \max\left\{  G_{0},G_{1}\right\}=\left\vert \alpha\right\vert\left(  t_{1}-t_{0}\right)  ^{2}
\]
the invariance of the ball $B\left(  0,R\right)  $ under the action of $S$ can
be reduced, for $B=\max\left\{ \left\vert b\right\vert,\left\vert \dot{b}\right\vert\right\}$ to the following condition%
\[
G_{2}\left(  R^{2}+2aR+a^2\right)  +B \leq R
\]
which can be satisfied by some $R>0$ if $2aG_2<1$ and
\[
\Delta=\left(  2aG_{2}-1\right)  ^{2}-4G_{2}\left(  G_{2}a^2+B \right)  = 1-4aG_2-4BG_2>0,
\]
or after rephrasing
\begin{equation}
4(a+B)G_2<1\,.\label{best}
\end{equation}
\end{proof}

\begin{corollary}
The condition $\left(  \ref{best}\right)$
is satisfied if $a$ and $\beta$ are sufficiently small.
\end{corollary}

\begin{example}
If $A=0,$ i.e. we neglect the drag force $\alpha=0$, accounting for negligible, in real
world model, part, our assumptions are naturally satisfied. In such a context
it is natural that one can attain the destination point in the desired span 
of time. In fact since the equation is linear one can derive the explicit
formula for the solution.
\end{example}

\subsection{Further developments. A priori estimates. Sign sensitive case.}

In this section we present some extensions to provide more subtle a priori estimates taking into regard
the negative sign of some terms appearing in $\left(  \ref{BVP}\right) $. Instead
of using Schauder Theorem, one can the try to apply Schafer Theorem yielding a fixed point as well.

First we neglect non-negative term in (\ref{BVP}) and integrate inequality
$$-\ddot{z}(t)\ge \beta (t)$$ using boundary conditions to obtain
$$(t-t^*)\dot{z}(t)\le -(t-t^*)\int_{t^*}^t\beta(s)ds $$
for any $t$, while $t^*$ chosen such that $\dot{z}(t^*)=0$ which exists due to the zero homogeneous
boundary conditions on imposed on the function $z$. 

Next one can integrate, after dividing 
by $\left(  \dot{z}+a\right)  ^{2}$,  both sides of (\ref{BVP}) to get better a priori bounds as follows
$$\dot{\psi}(t)\ge \beta(t) \psi^2(t)$$
where $\psi(t)=(\dot{z}(t)+a)^{-1}$. Hence for any $t$ after integrating the above inequality we have $$|\dot{z}(t)+a|(t-t^*)\le (t-t^*)(|a|-\int_{t^*}^t \beta(s)ds)\,.$$
Thus we can get by (\ref{Tso}) partial bounds for $\dot{z}$ and consequently for $z$. 

This approach we shall not fully exploit herein, since there is a gap in the above a priori estimates but it will be  addressed in the subsequent papers.
 
%Dividing by $\left(  \dot{z}+a\right)  ^{2}$...

%\section{Initial value problem}

\end{document}